\documentclass[12pt]{amsart}
\usepackage{enumitem}
\usepackage{amssymb}
\usepackage{tikz}
\usetikzlibrary{calc,tikzmark}
\usepackage{diagbox}
\usepackage{epic}
\usepackage{marginnote}
\usepackage{tcolorbox}
\usepackage[colorinlistoftodos]{todonotes}

\usepackage[xcolor]{mdframed}
\usepackage{marginnote}
\usepackage{tcolorbox}
\usepackage[colorinlistoftodos]{todonotes}
\usepackage{geometry}
\newcommand{\PP}{\mathbb{P}}

\newcommand{\KK}{\mathbb{K}}

\newcommand{\Pic}{\operatorname{Pic}}
\newcommand{\gggg}{\mathfrak{g}}
\newcommand{\hhh}{\mathfrak{h}}
\newcommand{\ttt}{\mathfrak{t}}
\newcommand{\sss}{\mathfrak{s}}

\newcommand{\uni}{\textnormal{uni}}
\newcommand{\red}{\textnormal{red}}
\newcommand{\sms}{\textnormal{ss}}

\DeclareMathOperator{\rk}{rk}

\newtheorem{theorem}{Theorem}[section]
\newtheorem{lemma}[theorem]{Lemma}

\theoremstyle{definition}

\newtheorem{example}[theorem]{Example}

\theoremstyle{remark}
\newtheorem{remark}[theorem]{Remark}

\theoremstyle{proposition}
\newtheorem{proposition}[theorem]{Proposition}

\theoremstyle{corollary}
\newtheorem{corollary}[theorem]{Corollary}

\theoremstyle{conjecture*}
\newtheorem*{conjecture*}{Conjecture}

\numberwithin{equation}{section}

\begin{document}

\title[On Picard number of homogeneous spaces]{On Picard number and dimension \\ of algebraic homogeneous spaces}



\author{Ivan Beldiev \and Dmitry Timashev}
\address{HSE University, Faculty of Computer Science, Pokrovsky Boulevard 11, Moscow, 109028 Russia}
\email{ivbeldiev@gmail.com, isbeldiev@hse.ru}
\address{Lomonosov Moscow State University, Faculty of Mechanics and Mathematics, Department of Higher Algebra, and Moscow Center of Fundamental and Applied Mathematics, 119991 Moscow, Russia}
\email{timashev@mccme.ru}

\thanks{The study was implemented in the framework of the Basic Research Program at the HSE University.}

\subjclass[2020]{Primary 14L30, 17B20; Secondary 20G07}

\keywords{Algebraic variety, affine variety, Picard group, algebraic group, reductive group, homogeneous space}

\date{}

\dedicatory{}

\begin{abstract}
    An algebraic variety $X$ is called a homogeneous space if there exists a transitive regular action of an algebraic group on $X$. We prove inequalities between the dimension of a homogeneous space of a linear algebraic group and its Picard number.
\end{abstract}

\maketitle

\section*{Introduction}

Let $\KK$ be an algebraically closed field of characteristic zero. All algebraic varieties in this paper are defined over $\KK$.

An algebraic variety $X$ is called a \emph{homogeneous space} if it admits a transitive regular action of a linear algebraic group $G$. In this case, $X$ can be identified with the variety of left cosets $G/H$, where $H$ is the stabilizer in $G$ of any point of $X$. Homogeneous spaces have rich structural theory and many applications; see, e.g., \cite{Br, Br2, Gr, Hu, OV, VP-I, Ti}.

In \cite{AZ}, the authors also consider the notion of \emph{homogeneous variety}, i.e., an algebraic variety~$X$ such that the automorphism group $\textnormal{Aut}(X)$ acts on $X$ transitively. In general, $\textnormal{Aut}(X)$ is not an algebraic group, so this notion does not coincide with that of homogeneous space by definition. In fact, there are examples of homogeneous varieties which are not homogeneous spaces. In \cite{AZ}, a series of such examples is given.

Denote by $\rho(X)$ the Picard number of an algebraic variety $X$, i.e., the rank of its Picard group $\textnormal{Pic}(X)$. One of the techniques used in \cite{AZ} is to compare the Picard number and the dimension of a variety. The authors prove that the Picard number of an irreducible affine homogeneous space cannot be greater than its dimension. As an application, several examples of affine homogeneous varieties not satisfying this inequality and thus not isomorphic to homogeneous spaces are given.

Relations between the Picard number and the dimension of an algebraic variety are also studied in other works. A famous example is the generalized Mukai conjecture proposed in~\cite{BCDD}. According to this conjecture, the inequality $\rho(X)\cdot (i_X - 1) \leq \dim X$ holds for Fano varieties, where $i_X$ is the pseudo-index of~$X$. In \cite{BCDD}, this inequality is proved for Fano varieties satisfying certain conditions. There are subsequent papers on the generalized Mukai conjecture; see \cite{Andr, Cas, Cast, Fu, Ga, Pas, Rei}.

It turns out that the inequality $\rho(X) \leq \dim X$ for affine homogeneous spaces $X \simeq G/H$ can be significantly strengthened in certain cases. In this paper, we prove stronger estimates for $\rho(X)$ depending on the group $G$.

The paper is organized as follows. In Section 1, we observe that the inequality $\rho(X) \leq \dim X$ proved in \cite{AZ} for affine homogeneous spaces can be extended to the non-affine case. In Section 2, we prove stronger inequalities for affine homogeneous spaces of reductive groups. Let a connected reductive group $G$ be decomposed into an almost direct product $Z\cdot G_1\cdot G_2\cdots G_m$ of a torus $Z$ and simple groups $G_i$ (see \cite[Chapter 10, Section 27.5]{Hu}). Then for any affine homogeneous space $X$ of $G$ we have
$$\rho(X) \leq \frac{\dim X}{\min \rk(G_i) + 1},$$
where $\rk(G_i)$ is the rank of $G_i$, while in the case $m = 0$, i.e., when $G$ is a torus, we have $\rho(X) = 0$. As a corollary, we prove the inequality $\rho(X) < \sqrt{\dim X}$, where $X$ is an affine homogeneous space of a simple algebraic group.

In Section 3, we  prove the following inequality for an arbitrary homogeneous space $X$ of a connected linear algebraic group $G$:
$$\rho(X) \leq \frac{2\dim X}{\min \rk(G_i) + 1},$$
where $G = (Z\cdot G_1\cdot G_2 \cdots G_m)\ltimes G_{\uni}$ is a Levi decomposition, the subgroups $Z, G_1, G_2, \ldots, G_m$ are the same as above, and $G_{\uni}$ is the unipotent radical of $G$; if $m = 0$, then $\rho(X) = 0$. Then we deduce the inequality $\rho(X) < \sqrt{2\dim X}$ for homogeneous spaces of simple algebraic groups.

Note that all projective homogeneous spaces $X$ are Fano varieties. For homogeneous Fano varieties, the generalized Mukai conjecture holds true, see \cite[Corollary~5.4]{BCDD}. However, our inequalities do not follow from this result. In Remark \ref{Mukai}, we provide an example when our inequalities give a stronger bound. Also, the methods we use are quite different as we work directly with algebraic groups not using geometric tools as in \cite{BCDD}.

\section*{Acknowledgements}
The first author is grateful to his academic supervisor Ivan Arzhantsev for posing the problem and useful discussions. We also wish to thank Cinzia Casagrande for valuable comments on the generalized Mukai conjecture for Fano varieties.

\section{The Picard number of a homogeneous space}

We start with the following proposition.

\begin{proposition}\label{gen_ieq_old}
    Let $X$ be an irreducible homogeneous space of a linear algebraic group. Then the inequality $\rho(X) \leq \dim X$ holds.
\end{proposition}

In the case when $X$ is an affine homogeneous space, Proposition \ref{gen_ieq_old} is proved in \cite[Lemma 2]{AZ}. We recall the proof for the reader's convenience, observing that the affinity assumption can be dropped provided that the group $G$ acting on $X$ is linear. First, we need the following lemma.

\begin{lemma}\label{torus_estimate}
    Let $X = G/H$ be an irreducible homogeneous space of a linear algebraic group $G$. Consider a Levi decomposition $H=H_{\red}\ltimes H_{\uni}$, where $H_{\red}$ is a Levi subgroup and $H_{\uni}$ is the unipotent radical of $H$. Let the identity component of $H_{\red}$ be decomposed into an almost direct product $H_{\red}^0=T\cdot S$ of a torus $T$ and a semisimple group $S$. Then $\rho(X)=\rho(G/H_{\red})\leq\rho(G/H^0)\leq \dim T$.
\end{lemma}

\begin{proof}
    Replacing $G$ with its identity component, we may assume that $G$ is connected. Moreover, there exists a central isogeny $\widehat{G} \to G$ such that $\Pic(\widehat{G}) = 0$ (see \cite[Theorem 3]{VP}). Replacing $G$ with $\widehat{G}$, we may assume that $G$ is connected with $\Pic(G) = 0$.

    By \cite[Corollary of Theorem 4]{VP}, the condition $\Pic(G) = 0$ implies that the group $\Pic(X)$ is isomorphic to $\mathfrak{X}(H)/\mathfrak{X}_G(H)$, where $\mathfrak{X}(H)$ is the group of characters of $H$ and $\mathfrak{X}_G(H)$ is the group of characters of $H$ that can be extended to a character of~$G$. Since unipotent groups have only trivial characters, the restriction map ${\mathfrak{X}(H) \to \mathfrak{X}(H_{\red})}$ is an isomorphism, which implies that $\Pic(X)\simeq\Pic(G/H_{\red})$ and $\rho(X)=\rho(G/H_{\red}) \leq \operatorname{rk}\mathfrak{X}(H)$.

    The characters of $H$ that are trivial on $H^0=H_{\red}^0\ltimes H_{\uni}$ are the characters of the finite group $H/H^0$, hence the restriction homomorphism $\mathfrak{X}(H) \to \mathfrak{X}(H^0)$ has finite kernel. It follows that $\rho(X) \leq \rho(G/H^0)$.

    Since $H_{\red}^0 = T\cdot S$ and semisimple groups have only trivial characters, the restriction homomorphism $\mathfrak{X}(H^0) \to \mathfrak{X}(T)$ is injective. Combining this with the above inequalities, we get $\rho(X) \leq \rho(G/H^0) \leq \operatorname{rk}\mathfrak{X}(H^0) \leq \operatorname{rk}\mathfrak{X}(T) = \dim T$, as desired.
\end{proof}

\begin{proof}[Proof of Proposition \ref{gen_ieq_old}]
    Replacing $G$ with the quotient by the non-effectivity kernel of the action, we may assume that $G$ acts on $X$ effectively. In the notation of Lemma~\ref{torus_estimate}, the torus $T\subset G$ also acts effectively, whence locally freely on $X$. It follows that $\rho(X) \leq \dim T \leq \dim X$, as desired.
\end{proof}

\begin{example}\label{first_example}
    There are examples of homogeneous spaces satisfying the equality $\rho(X) = \dim X$ in any dimension. Indeed, the projective line $\PP^1 \simeq SL(2,\KK)/B$, where $B$ is the Borel subgroup of upper-triangular matrices, is a homogeneous space. This implies that $X_n = (\PP^1)^n$ is also a homogeneous space with $\dim X_n = \rho(X_n) = n$.
\end{example}

\section{Affine homogeneous spaces of reductive groups}


Let us start with the case when $G$ is a simple algebraic group. We prove the following proposition.

\begin{proposition}\label{affine_1}
    Let $X$ be an affine homogeneous space of positive dimension of a simple algebraic group $G$. Then the following inequality holds:
$$\rho(X) \leq \frac{\dim X}{1 + \rk G}.$$
\end{proposition}

\begin{proof}
    By the Matsushima criterion (see, for example, \cite{A_Mats} or \cite[Section 4.7]{VP-I}), $X$ is isomorphic to the quotient $G/H$, where $H$ is a reductive subgroup of $G$. Denote by $T$ the connected center of the identity component $H^0$ of~$H$. As we know from Lemma \ref{torus_estimate}, $\rho(X) \leq \dim T$. Denote by $L$ the centralizer of $T$ in~$G$. Note that $L$ is a Levi subgroup in a certain parabolic subgroup of $G$ (see \cite[Th{\'e}or{\`e}me 4.15]{BT}) and $T$ is contained in the connected center $\widehat{T}$ of $L$, so $\dim T \leq \dim \widehat{T}$. Since $H^0 \subset L$, we have $\dim X \geq \dim G/L$. Therefore it suffices to prove the proposition for $H = L$ and $T=\widehat{T}$.

    Since $L$ is reductive, it decomposes into an almost direct product $L = T \cdot L_1\cdots L_m$, where $L_i$ are simple subgroups. Denote $r = \rk G$, $r_i = \rk L_i$ ($i = 1, 2, \ldots, m$), and $s = \dim T$. Clearly, we have $r = s + r_1 + r_2 + \dots + r_m$. Also, denote by $N$ and $N_i$ the number of roots of $G$ and $L_i$ respectively; we have $\dim G = r + N$, $\dim L_i = r_i + N_i$, and $\dim T = s$. It follows that $\dim X = N - N_1 - N_2 - \dots - N_m$.

    We have $N = rh$ and $N_i = r_ih_i$, where $h$ and $h_i$ are the Coxeter numbers of $G$ and $L_i$ respectively (see \cite[Chapter 5, Theorem 6.2.1]{Bour}). It is also proved in \cite[Chapter~6, Proposition 1.11.31]{Bour} that $h - 1$ is the sum of coefficients of the linear expression of the highest root of $G$ in the basis of simple roots, and the same is true for $h_i - 1$. Since the positive, resp. simple roots of $L_i$ are among the positive, resp. simple roots of $G$, this implies the inequality $h_i \leq h$ for all $i = 1, 2, \ldots, m$. Moreover, since all coefficients of this linear expression are positive integers, it follows that $h \geq r + 1$. Thus we have
    $$\dim X = rh - \sum_{i = 1}^{m}r_ih_i \geq \left(r - \sum_{i=1}^{m}r_i\right)h = sh \geq s(r + 1) \geq \rho(X)\cdot (\rk G + 1),$$
    which is exactly the inequality we are proving.
\end{proof}

Now we proceed to the general case of a reductive group $G$. Any connected reductive group $G$ is isomorphic to an almost direct product $Z\cdot G_1\cdot G_2 \cdots G_m$ of its connected center $Z$ and simple groups $G_i$ ($1\leq i \leq m$). We have the following theorem.

\begin{theorem}\label{affine_ss}
    If $G = Z\cdot G_1\cdot G_2 \cdots G_m$ is a connected reductive group (where $Z$ is the connected center and $G_i$ are the simple factors of $G$) and $X$ is an affine homogeneous space of $G$, then we have
    $$\rho(X) \leq \frac{\dim X}{1 + \min{\rk G_i}}.$$
    If $m = 0$, then $\rho(X) = 0$.
\end{theorem}

\begin{proof}
    By the Matsushima criterion, $X=G/H$ with $H$ reductive. First, let us show that $G$ may be assumed semisimple. Denote the semisimple group $G_1\cdot G_2\cdots G_m$ by~$G_{\sms}$. By Lemma \ref{torus_estimate}, we may assume without loss of generality that $G = Z \times G_{\sms}$ and $H$ is connected.


    The image of $H$ under projection to $Z$ is a subtorus $Z_0\subset Z$. Denote by $Z_1$ any subtorus in $Z$ complementary to $Z_0$. Since $Z = Z_1 \times Z_0$ and $Z_0 \times G_{\sms} = G_{\sms} \cdot H$, the homogeneous space $X$ is isomorphic to
    $$G/H = Z_1 \times (Z_0 \times G_{\sms})/H \simeq Z_1\times G_{\sms}/(H\cap G_{\sms}).$$

    Now consider the homogeneous space $\widetilde{X} = G_{\sms}/(H\cap G_{\sms})$. We have $X \simeq Z_1\times\widetilde{X}$. Clearly, $\dim X \geq \dim{\widetilde{X}}$ and $\Pic(X) = \Pic(\widetilde{X})$, since $Z_1$ is a torus and $\Pic(Z_1)$ is trivial. Therefore we can replace $G$ and $H$ with $G_{\sms}$ and $H\cap G_{\sms}$ respectively.

    Let us now assume that $G$ is semisimple and $H = T\cdot S$ is an almost direct product of a torus $T$ and a semisimple subgroup~$S$. Passing to the corresponding Lie algebras, we obtain the decomposition $$\gggg = \gggg_1 \oplus \gggg_2 \oplus \dots \oplus \gggg_m,$$
    where $\gggg$ and $\gggg_i$ are the Lie algebras of $G$ and $G_i$ respectively. The Lie algebra $\hhh$ of $H$ is of the form $\hhh = \ttt \oplus \sss$, where $\ttt$ is the center of $\hhh$ and $\sss$ is semisimple. Denote by $\pi_i$ the projection of $\hhh$ to $\gggg_i$. Then $\pi_i(\hhh)$ is a reductive Lie subalgebra of $\gggg_i$ and $\pi_i(\hhh) = \pi_i(\ttt) \oplus \pi_i(\sss)$ is its decomposition into the direct sum of the center $\pi_i(\ttt)$ and the semisimple Lie algebra $\pi_i(\sss)$. It follows from Proposition \ref{affine_1} that $$\dim\pi_i(\ttt) \leq \frac{\dim\gggg_i - \dim\pi_i(\hhh)}{1 + \rk \gggg_i} \leq \frac{\dim\gggg_i - \dim\pi_i(\hhh)}{1 + \min{\rk \gggg_i}}.$$

    Now we have
    $$\dim\ttt \leq \sum_{i=1}^{m} \
    \dim\pi_i(\ttt) \quad \text{and} \quad \dim\hhh \leq \sum_{i=1}^{m}\dim\pi_i(\hhh),$$
    hence, by Lemma \ref{torus_estimate},
    \begin{multline*}
      \rho(X) \leq \dim\ttt\leq\sum_{i=1}^{m}\dim\pi_i(\ttt) \leq \sum_{i=1}^{m}\frac{\dim \gggg_i - \dim\pi_i(\hhh)}{1 + \min\rk\gggg_i} \le \\
      \leq \frac{\dim\gggg - \dim\hhh}{1 + \min\rk\gggg_i} = \frac{\dim X}{1 + \min\rk\gggg_i},
    \end{multline*}
    as desired.

    In the case $m = 0$, i.e., when $G$ is a torus, the homogeneous space $X$ is also a torus. Therefore $\rho(X) = 0$.
\end{proof}

\begin{corollary}\label{ineq_ss_cor}
    If $X$ is an affine homogeneous space of positive dimension of a reductive group, then $$\rho (X) \leq \frac{1}{2}\dim X.$$
\end{corollary}
\begin{proof}
This follows immediately from Theorem \ref{affine_ss} and the fact that the rank of any simple group is at least $1$. In the case when the decomposition of $G$ is a torus, we have $\rho(X) = 0$ and the inequality in the statement also holds.
\end{proof}

If the acting group $G$ is simple, then we deduce from Proposition \ref{affine_1} a stronger inequality between $\rho(X)$ and $\dim X$ that does not depend on $\rk G$.

\begin{corollary}\label{aff_sqrt}
    Let $X$ be an affine homogeneous space of a simple group $G$ of positive dimension. Then $\rho(X) < \sqrt{\dim X}$.
\end{corollary}

\begin{proof}
    If $X = G/H$ for a reductive algebraic subgroup $H\subset G$, then $\rho(X) \leq \dim T$  by Lemma \ref{torus_estimate}, where $T$ is the connected center of $H$. Since $\rk G$ is the dimension of a maximal torus in $G$, we have $\dim T \leq \rk G < \rk G + 1$. If $\rho(X) \ne 0$, then this implies $$\rho(X)^2 < (\rk G + 1)\rho(X) \leqslant \dim X$$ by Proposition \ref{affine_1}, so $\rho(X) < \sqrt{\dim X}$ as desired. For $\rho(X) = 0$, the inequality in the statement is obvious.
\end{proof}

\begin{example}
    Let $X = G/H$ be an affine homogeneous space of positive dimension of a simply connected semisimple group $G$. It follows from the proof of Theorem~\ref{affine_ss} that the inequality from Corollary \ref{ineq_ss_cor} turns into equality if and only if $G$ is the direct product of $n$ copies of the special linear group~$SL_2$ and $H = T^n$, where $T\subset SL_2$ is the subgroup of diagonal matrices. It is known that $SL_2/T$ is the non-degenerate smooth affine quadric $Q_2$ of dimension~$2$, hence $X \simeq Q_2^n$ with $\dim X = 2n$ and $\rho(X) = n$. In particular, we see that for any even positive integer $2n$ there exists a unique affine homogeneous space of a semisimple group of dimension $2n$ and Picard number $n$.
\end{example}

\section{Homogeneous spaces of linear algebraic groups: the general case}

In this section, we consider the general case when $X$ is a homogeneous space of an arbitrary linear algebraic group $G$. We start with the following lemma.

\begin{lemma}\label{isotropic}
    Let $G$ be a reductive group and $H\subset G$ a closed subgroup with a Levi decomposition $H=H_{\red}\ltimes H_{\uni}$, where $H_{\red}$ is a reductive subgroup and $H_{\uni}$ is the unipotent radical of $H$. Then
    $$\dim G/H \geqslant \frac{1}{2}\dim G/H_{\red}.$$
\end{lemma}

\begin{proof}
    Denote the Lie algebras of $G$, $H$, $H_{\red}$, $H_{\uni}$ by $\gggg$, $\hhh$, $\hhh_{\red}$, $\hhh_{\uni}$, respectively. Consider an invariant symmetric bilinear form $\beta$ on $\gggg$ defined by the formula
    $$\beta(\xi,\eta)=\operatorname{tr}\bigl(dR(\xi)\cdot dR(\eta)\bigr),\qquad\xi,\eta\in\gggg,$$
    where $R$ is a faithful linear representation of $G$. Since $G$ is reductive, the form $\beta$ is non-degenerate. It is also known that the kernel of the restriction of $\beta$ to $\hhh$ coincides with $\hhh_{\uni}$ and, in particular, the restriction of $\beta$ to $\hhh_{\red}$ is non-degenerate (see, e.g., \cite[Chapter 4, Section 1]{OV}). It follows that the restriction of $\beta$ to the orthogonal complement $\hhh_{\red}^\perp$ is also non-degenerate and $\hhh_{\uni} \subset \hhh_{\red}^\perp$ is an isotropic subspace. This implies the inequality $\dim \hhh_{\uni} \leq \frac{1}{2}\dim \hhh_{\red}^\perp = \frac{1}{2}\dim G/H_{\red}$. Consequently,
    \begin{equation*}
    \dim G/H = \dim G/H_{\red} - \dim H_{\uni} \geq \frac{1}{2}\dim G/H_{\red}.\qedhere
    \end{equation*}
\end{proof}

Now we are ready to prove the following theorem.

\begin{theorem}\label{gen}
    Let $X$ be a homogeneous space of positive dimension of a connected linear algebraic group~$G$ with a Levi decomposition $G=(Z\cdot G_1\cdot G_2\cdots G_m)\ltimes G_{\uni}$, where $Z$ is a torus, $G_i$ are simple groups, and $G_{\uni}$ is the unipotent radical of $G$. Then the following inequality holds:
    $$\rho(X) \leq \frac{2\dim X}{\min \rk(G_i) + 1}$$
    If $m = 0$, then $\rho(X) = 0$.
\end{theorem}

\begin{proof}
    First, we show that $G$ may be assumed reductive. Consider the Levi subgroup $G_{\red}=Z\cdot G_1\cdot G_2\cdots G_m$ of $G$. Let $X = G/H$. Denote $\widetilde H = HG_{\uni}$ and $\widetilde X = G/\widetilde H$. It is clear that $\dim \widetilde X \leq \dim X$. Moreover, $H$ and $\widetilde{H}$ share one and the same Levi subgroup $H_{\red}$. By Lemma~\ref{torus_estimate}, $\rho(X) = \rho(\widetilde X)$. Therefore it suffices to prove the inequality in the statement for $G/\widetilde{H}$, which is a homogeneous space for $G_{\red}$. Thus we may assume without loss of generality that $G$ is reductive.

    As before, we have $H = H_{\red}\ltimes H_{\uni}$. By Lemma \ref{isotropic}, $\dim X \geq \frac{1}{2}\dim G/H_{\red}$. On the other side, by Lemma~\ref{torus_estimate} and Theorem \ref{affine_ss},
    $$\rho(X) = \rho(G/H_{\red}) \leq \frac{\dim (G/H_{\red})}{1 + \min \rk G_i} \leq \frac{2\dim X}{1 + \min\rk G_i},$$
    as desired. If $m = 0$, then $\rho(G/H_{\red}) = 0$ again by Theorem \ref{affine_ss}, hence $\rho(X) = 0$.
\end{proof}

Similarly to the affine case, we have the following estimate for the Picard number of a homogeneous space of a simple group.

\begin{corollary}\label{gen_s}
    Let $X$ be a homogeneous space of a simple algebraic group. Then
    $$\rho(X) < \sqrt{2\dim X}.$$
\end{corollary}

\begin{proof}
    Let $X = G/H$, where $H$ has a Levi decomposition $H=H_{\red}\ltimes H_{\uni}$, as above. By Lemma \ref{isotropic}, we have $\dim X \geq \frac{1}{2} \dim G/H_{\red}$. Then Corollary \ref{aff_sqrt} applied to the affine variety $G/H_{\red}$ implies the desired inequality.
\end{proof}

\begin{remark}
    Let $X = G/H$ be a projective homogeneous space of positive dimension of a simply connected semisimple group $G$ acting locally effectively. It follows from the proof of Theorem \ref{gen} that the inequality $\rho(X) \leq \dim X$ turns into equality if and only if $G$ is a direct product of $n$ copies of the special linear group $SL_2$ and $X \simeq (\PP^1)^n$ with $\dim X = \rho(X) = n$; cf.\ Example~\ref{first_example}.
\end{remark}

\begin{remark}\label{Mukai}

Note that the inequalities we proved do not follow from the generalized Mukai conjecture. For example, take $G = SL_n$ and $B\subset G$ a Borel subgroup. In this case, the variety $X = SL_n/B$ is the complete flag variety and it is known that the pseudoindex $i_X$ of $X$ is equal to $2$. So, the generalized Mukai conjecture implies the inequality $\rho(X) \leq \dim(X)$ while Theorem \ref{gen} gives a stronger bound, namely $\rho(X) \leq \frac{2}{n}\dim X$. The latter bound is in fact optimal in this case since $\dim X = \frac{n(n-1)}{2}$ and $\rho(X) = n-1$.

\end{remark}










\begin{thebibliography}{99}

\bibitem{Andr} M. Andreatta, E. Chierici, and G. Occhetta. Generalized Mukai conjecture for special Fano varieties. Centr. Eur. J. Math. 2 (2004), 272--293.

\bibitem{A_Mats} I. Arzhantsev. Invariant ideals and Matsushima's criterion. Comm. Algebra 36 (2008), no. 12, 4368--4374.

\bibitem{AZ} I. Arzhantsev and Yu. Zaitseva. Affine homogeneous varieties and suspensions. Res. Math. Sci. 11 (2024), no. 2, article 27.

\bibitem{BCDD} L. Bonavero, C. Casagrande, O. Debarre, and S. Druel. Sur une conjecture de Mukai. Comment. Math. Helv. 78 (2003), 601--626. 

\bibitem{BT} A. Borel and J. Tits. Groupes r\'{e}ductifs. Inst. Hautes \'{E}tudes Sci. Publ. Math. 27 (1965), 55--150.


\bibitem{Bour} N. Bourbaki. Groupes et alg\`{e}bres de Lie: Chapitres 4, 5, et 6. Actualites Scientifiques et Industrielles, Hermann, Paris, 1968.

\bibitem{Br} M. Brion. Some structure theorems for algebraic groups. In: Algebraic Groups: Structure and Actions, Mahir Bilen Can (Editor). Proceedings of Symposia in Pure Mathematics 94 (2017), 53--126.

\bibitem{Br2} M. Brion, P. Samuel, and V. Uma. Lectures on the Structure of Algebraic Groups and Geometric Applications. Hindustan Book Agency, New Delhi, 2013.

\bibitem{Cas} C. Casagrande. The number of vertices of a Fano polytope. Ann. Inst. Fourier (Grenoble) 56 (2006), 121--130.

\bibitem{Cast} A. Caviedes Castro, M. Pabiniak, and S. Sabatini. Generalizing the Mukai conjecture to the symplectic category and the Kostant game. Pure Appl. Math. Q. 19 (2023), 1803--1837.



\bibitem{Fu} K. Fujita. The generalized Mukai conjecture for toric log Fano pairs. European Journal of Mathematics 5 (2019), 858--871.

\bibitem{Ga} G. Gagliardi and J. Hofscheier. The generalized Mukai conjecture for symmetric varieties. Trans. Amer. Math. Soc. 369 (2017), 2615--2649.

\bibitem{Gr} F. Grosshans. Algebraic Homogeneous Spaces and Invariant Theory. Lect. Notes in Math., Springer-Verlag, Berlin, 1997.

\bibitem{Hu} J. Humphreys. Linear Algebraic Groups. Grad. Texts in Math. 21, Springer-Verlag, New York, 1975.



\bibitem{OV} A. Onishchik and E. Vinberg. Lie Groups and Algebraic Groups. Springer Series in Soviet Mathematics, Springer, Berlin, 1990.

\bibitem{Pas} B. Pasquier. The pseudo-index of horospherical Fano varieties. Internat. J. Math. 21 (2010), 1147--1156.

\bibitem{VP} V. Popov. Picard groups of homogeneous spaces of linear algebraic groups and one-dimensional homogeneous vector bundles. Math. USSR-Izv. 8 (1974), no. 2, 301--327.

\bibitem{VP-I} V. Popov and E. Vinberg. Invariant Theory. In: Algebraic Geometry IV, A.N. Parshin, I.R. Shafarevich (Editors), Springer-Verlag, Berlin, Heidelberg, New York, 1994.

\bibitem{Rei} M. Reineke. The Mukai conjecture for Fano quiver moduli. Algebr. Represent. Theory 27 (2024), 1641--1644.

\bibitem{Ti} D. Timashev. Homogeneous Spaces and Equivariant Embeddings. Encyclopaedia Math. Sciences 138, Springer-Verlag, Berlin, Heidelberg, 2011.


\end{thebibliography}
\end{document}